\documentclass[english]{ourlematema}
\usepackage{amsmath, amsthm, amssymb, hyperref, color}
\usepackage{MnSymbol} % for \dashedrightarrow command
\usepackage{graphicx}
\usepackage{caption}
\usepackage[all]{xypic}
\usepackage{verbatim}
\usepackage{chemfig, chemnum}
\usepackage{tikz}
\usepackage[linesnumbered,lined,commentsnumbered]{algorithm2e}
\usepackage{caption}
\usepackage{subcaption}
\usepackage{makecell}
\usepackage{array}

\newtheorem{theorem}{Theorem}
 \numberwithin{theorem}{section}
\newtheorem{proposition}[theorem]{Proposition}

\newtheorem{remark}[theorem]{Remark}
\newtheorem{example}[theorem]{Example}
\newtheorem{conjecture}[theorem]{Conjecture}

\theoremstyle{definition}

\newcommand{\xp}{\mathfrak{p}}
\newcommand{\xg}{\mathfrak{g}}
\newcommand{\xq}{\mathfrak{q}}
\newcommand{\PP}{\mathbb{P}}
\newcommand{\RR}{\mathbb{R}}
\newcommand{\QQ}{\mathbb{Q}}

\newcommand{\ZZ}{\mathbb{Z}}

 \title{An Octanomial Model for Cubic Surfaces}
 
 \author{Marta Panizzut}
 \address{%
Institut f\"ur Mathematik, 
TU Berlin \\
\email{panizzut@math.tu-berlin.de}
}
\author{Emre Can Sert\"oz}
\address{%
MPI for Mathematics in the Sciences, Leipzig \\
\email{emresertoz@gmail.com}
}

\author{ Bernd Sturmfels}
\address{%
MPI for Mathematics in the Sciences, Leipzig \\
\email{bernd@mis.mpg.edu }
}

\begin{document}
\maketitle
\begin{abstract}
\noindent
We present a new normal form for cubic surfaces that is well suited for $p$-adic geometry, 
as it reveals the intrinsic del~Pezzo combinatorics of the $27$ trees in the tropicalization.
The new normal form is a polynomial with eight terms, written in moduli from the ${\rm E}_6$ hyperplane arrangement. If such~a surface is tropically smooth then its $27$ tropical lines are distinct.
We focus on explicit computations, both
symbolic and $p$-adic numerical.
\end{abstract}

\section{Introduction}

Any configuration of six distinct points in the projective plane $\PP^2$
lies on a cuspidal cubic. Thus, after an automorphism of $\PP^2$,
 the homogeneous  coordinates of the points are the
columns of a $3 \times 6$ matrix that has the following special form:
\begin{equation}
\label{eq:sixpoints}
\begin{pmatrix} 
 1 & 1 & 1 & 1 & 1 & 1 \\
d_1 & d_2 & d_3 & d_4 & d_5 & d_6 \smallskip \\
d_1^3 & d_2^3 & d_3^3 & d_4^3 & d_5^3 & d_6^3
\end{pmatrix}.
\end{equation}
The $3 \times 3$-minors of this matrix factor into linear factors, and so does the condition for the six points to lie on a conic. The linear forms are
$\,d_i - d_j\,$ and $\, d_i+d_j+d_k \,$ and $\, d_1+d_2+d_3+d_4+d_5+d_6$, for a total of $36 =15+20+1$.
These define the $36$ hyperplanes in the reflection arrangement of  Coxeter type ${\rm E}_6$.
The complement of this hyperplane arrangement uniformizes
the $4$-dimensional moduli space of cubic surfaces.
This representation of marked del Pezzo surfaces of degree three
is used in many sources, including
\cite{CGL, CD, HKT, RSS}.

In this paper we study the cubic surface in $\PP^3$ that is obtained by blowing up
$\PP^2$ at the six points in (\ref{eq:sixpoints}). We write this cubic explicitly in terms of 
its moduli parameters $d_i$. To this end,
we choose the following specific basis for the four-dimensional space of ternary cubics 
that vanish on our six points in $\PP^2$:
\begin{equation}
\label{eq:factoredbasis}
x = F_{12} F_{34} F_{56} \, , \quad
y = F_{13} F_{25} F_{46} \, , \quad
z = F_{12} F_{35} F_{46} \, , \quad
w = F_{13} F_{24} F_{56}.
\end{equation}
The factors are linear forms that vanish on pairs of points. 
Explicitly, these are
\begin{equation}
\label{eq:lineinplane} \qquad F_{ij} \,\,= \,\,
 d_i d_j (d_i+d_j) \cdot X \,-\, (d_i^2+d_i d_j+d_j^2) \cdot Y \,+\, Z
 \quad {\rm for} \,\, 1 \leq i < j \leq 6.
\end{equation}
As  in (\ref{eq:factoredbasis}) and (\ref{eq:lineinplane}), we shall write
$(X:Y:Z)$ and $(x:y:z:w)$ for the homogeneous coordinates on
$\PP^2$ and $\PP^3$ respectively. The resulting cubic surface is defined by
\begin{equation}
\label{eq:octanomial}
a \cdot xyz \,\,+\,
b \cdot xyw \,\,+\,
c \cdot xzw\,\,+\,
d \cdot yzw\,\,+\,
 e \cdot x^2 y \,\,+\,\,
f \cdot x y^2 \,\,+ \,\,
g \cdot z^2 w \,\,+\,\,
h \cdot z w^2 .
\end{equation}
The  coefficients $a,b,\ldots,h$ of this octanomial are quintics
 in the moduli parameters $d_1,\ldots,d_6$.
These are displayed in Proposition~\ref{prop:coeffs}.
The {\em support}  of~(\ref{eq:octanomial}) equals
  \begin{equation} \label{eq:support} \mathcal{A} \,= \,
\bigl\{ (1110), (1101), (1011), (0111), (2100), (1200), (0021), (0012) \bigr\}.
\end{equation}
The symmetry group of the Newton polytope ${\rm conv}(\mathcal{A})$ is isomorphic to $(\ZZ/2\ZZ)^3$.
It  acts by compatibly
  permuting the six points in $\mathbb{P}^2$ and the coordinates in $\PP^3$.
   
\begin{figure}[ht]
  \centering
  \includegraphics{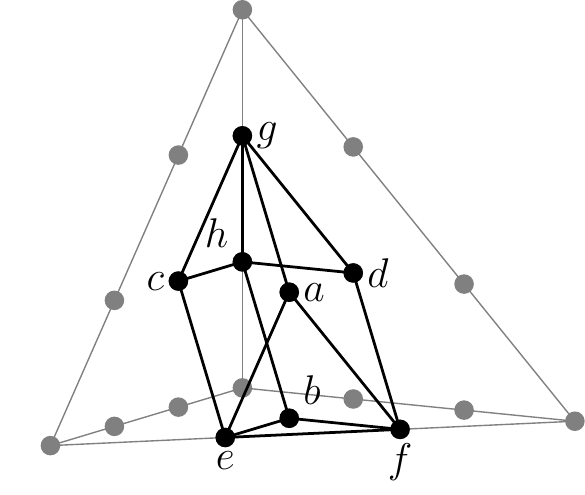}
  \caption{The Newton polytope ${\rm conv}(\mathcal{A})$ of the octanomial model. }
  \label{polytope}
\end{figure}
   
We propose the octanomial (\ref{eq:octanomial}) as a new normal form for cubic surfaces. 
Our study was inspired by the recent work of
Cueto and Deopurkar~\cite{CD}. They map
cubic surfaces from $\PP^3$ into $\PP^{44}$ by the linear forms
of the $45$ tritangent planes. They prove that 
 this embedding reveals the arrangement of $27$ trees
determined in \cite{RSS}. This raises the following question: {\em Which of the
coordinate projections $\,\PP^{44} \dashrightarrow \PP^3\,$
best preserve the  tropical line arrangement?}
We examined all $\,\binom{45}{4} =148,995 \,$ possibilities.
Among those with smallest support, we chose
the projection in~(\ref{eq:factoredbasis}), and we embarked on
the detailed study presented here.

For our octanomial surfaces, 
tropical smoothness imposes a striking 
constraint on the field of definition of the $27$ lines  (Theorem~\ref{thm:distinct_lines}). 
We believe that the same constraint also holds for cubics with full support (Conjecture~\ref{conj:kristin}).

The material that follows is organized into three sections.
In Section~\ref{sec2} we write the 
octanomial cubic~(\ref{eq:octanomial}) in terms of the moduli $d_1,\ldots,d_6$,
we compute its discriminant, and
we classify the unimodular triangulations of its Newton polytope (Theorem~\ref{thm:triangulations}).
The universal Fano scheme for the octanomial model,
described in Proposition~\ref{prop:ufv}, shows
that the $27$ lines are grouped into $15$ clusters.

Section~\ref{sec3} concerns the del Pezzo combinatorics
of tropically smooth cubic surfaces over $p$-adic fields. Tropicalization takes
their $27$ lines to arrangements of $27$ 
distinct trees. We study these arrangements and how they 
correspond to the different types of surfaces.
In addition to the stable arrangements in \cite{RSS},
we encounter some new types. These confirm
results on non-stable surfaces in~\cite{CD}.

Section~\ref{sec4} is independent of the earlier sections.
Here the focus is not on the octanomial model but we study
 general dense cubic surfaces~(\ref{eq:cubicf}) in $\PP^3$.
Based on
Theorems~\ref{thm:project} and~\ref{thm:eleven}, we offer
algorithms for computing their intrinsic structure
as a del Pezzo surface over a field with valuation.
Example \ref{ex:type7} shows the construction of dense cubics
that are tropically smooth, Naruki general and have 
their six points in $\PP^2$ over $\QQ$.
This addresses Question~11 from the {\em Twenty-seven Questions about the 
Cubic Surface} \cite{RS19}.
We conclude with a method that transforms a general cubic
into octanomial form~(\ref{eq:octanomial}), by identifying its 
moduli coordinates $d_1,\ldots,d_6$.

This paper is heavily computational. It relies on numerous implementations
and experiments with the software
 {\tt Macaulay2} \cite{M2}, {\tt Magma} \cite{Magma} and {\tt Maple} \cite{Maple}.
Our codes and our data are posted at our supplementary 
materials website\footnote{\url{https://software.mis.mpg.de/octanomial}}.

\section{Algebra and Combinatorics} \label{sec2}

We begin by presenting explicit formulas for working with the octanomial model.

\begin{proposition} \label{prop:coeffs}
The eight coefficients of the cubic surface (\ref{eq:octanomial}) defined by~(\ref{eq:factoredbasis})~are
$$ \footnotesize  \begin{matrix}
 a   \,= &  d_1  d_3  d_2  d_4 ( d_1+ d_3- d_2- d_4)+ d_2  d_4  d_5  d_6 ( d_2+ d_4- d_5- d_6)+ d_5  d_6  d_1  d_3 ( d_5+ d_6- d_1- d_3) + \\ 
&      \!\!\! \!\!\!\!\!\! \!\! d_5  d_6 ( d_5{+} d_6) ( d_1^2{+} d_3^2 {-} d_2^2  {-} d_4^2)
+ d_1  d_3 ( d_1{+} d_3) ( d_2^2{+} d_4^2{-} d_5^2{-} d_6^2)
+ d_2  d_4 ( d_2{+} d_4) ( d_5^2{+} d_6^2{-} d_1^2{-} d_3^2) \smallskip \\
 b  \,= &  d_1  d_2  d_3  d_5 ( d_1+ d_2- d_3- d_5)+ d_3  d_5  d_4  d_6 ( d_3+ d_5- d_4- d_6)+ d_4  d_6  d_1  d_2 ( d_4+ d_6- d_1- d_2) + \\
&     \!\!\! \!\!\!  \!\!\! \!\! d_4  d_6 ( d_4{+} d_6) ( d_1^2{+} d_2^2{-} d_3^2{-} d_5^2)+ d_1  d_2 ( d_1{+} d_2) 
( d_3^2{+} d_5^2{-} d_4^2{-} d_6^2)+ d_3  d_5 ( d_3{+} d_5) (d_4^2{+} d_6^2{-} d_1^2{-} d_2^2)
\smallskip \\
 c \, = &  d_1  d_3  d_2  d_5 ( d_1+ d_3- d_2- d_5)+ d_2  d_5  d_4  d_6 ( d_2+ d_5- d_4- d_6)+ d_4  d_6  d_1  d_3 ( d_4+ d_6- d_1- d_3) + \\
&      \!\!\! \!\!\! \!\!\! \!\! d_4  d_6 ( d_4{+} d_6) ( d_1^2{+} d_3^2{-} d_2^2{-} d_5^2)+ d_1  d_3 ( d_1{+} d_3)
 ( d_2^2{+} d_5^2 {-} d_4^2{-} d_6^2)+ d_2  d_5 ( d_2{+} d_5) ( d_4^2{+} d_6^2{-} d_1^2{-} d_3^2)
 \smallskip \\
  d  \,= &  d_1  d_2  d_3  d_4 ( d_1+ d_2- d_3- d_4)+ d_3  d_4  d_5  d_6 ( d_3+ d_4- d_5- d_6)+ d_5  d_6  d_1  d_2 ( d_5+ d_6- d_1- d_2) + \\
  &    \!\!\!  \!\!\! \!\!\! \!\! d_5  d_6 ( d_5{+} d_6) ( d_1^2{+} d_2^2{-} d_3^2{-} d_4^2)
  + d_1  d_2 ( d_1{+} d_2) ( d_3^2{+} d_4^2{-} d_5^2{-} d_6^2)
  + d_3  d_4 ( d_3{+} d_4) ( d_5^2{+} d_6^2{-} d_1^2{-} d_2^2)
  \smallskip \\
  e  \,= & -( d_1+ d_3+ d_5) ( d_2+ d_4+ d_6) ( d_1- d_5) ( d_2- d_6) ( d_3- d_4)
\qquad \qquad \qquad \qquad \qquad \qquad \qquad \quad \, \smallskip \\
 f  \,= & -( d_1+ d_2+ d_4) ( d_3+ d_5+ d_6) ( d_1- d_4) ( d_2- d_5) ( d_3- d_6)
\qquad \qquad \qquad \qquad \qquad \qquad \qquad \quad \, \smallskip \\
 g \, = & -( d_1+ d_3+ d_4) ( d_2+ d_5+ d_6) ( d_1- d_4) ( d_2- d_6) ( d_3- d_5)
 \qquad \qquad \qquad \qquad \qquad \qquad \qquad \quad \,\smallskip \\
h \, = & -( d_1+ d_2+ d_5) ( d_3+ d_4+ d_6) ( d_1- d_5) ( d_3- d_6) ( d_2- d_4)
\qquad \qquad \qquad \qquad \qquad \qquad \qquad \quad \,
\end{matrix} 
$$
\end{proposition}

\begin{proof}[Proof and Discussion]
This was found by a calculation in {\tt Macaulay2} over the field
$K = \mathbb{Q}(d_1,d_2,d_3,d_4,d_5,d_6)$, carried out
with the help of Mike Stillman. According to the 
Hilbert--Burch Theorem, the cubics (\ref{eq:factoredbasis}) that
cut out the six points in $\PP^2$ are maximal minors
of a  $3 \times 4$ matrix with entries in $K[X,Y,Z]_1$. We computed this $ 3 \times 4$ matrix
and we rearranged it to
 a $3 \times 3$ matrix with entries in $K[x,y,z,w]_1$ whose
 determinant is the octanomial  (\ref{eq:octanomial}) with the above coefficients.
 This  determinantal representation of our cubic, as well as explicit
 formulas for all $27$ lines over $K$, are posted on our supplementary
 website.
\end{proof}

\begin{remark} \rm The eight coefficients in Proposition \ref{prop:coeffs} sum to zero.
The image of the parametrization above is a quartic hypersurface in the hyperplane
$\PP^6 \subset \PP^7$ given by $\,a+b+c+d+e+f+g+h=0$. The equation of that
hypersurface in $\PP^6$ has $134$ terms when written in the first seven coordinates.
It equals $\, 4a^3b+4a^3e+4a^3f+9a^2b^2+4a^2bc+4a^2bd+20a^2be +
\cdots  +13b^2e^2+ \cdots +4f^4+4f^3g +4f^2g^2 = 0 $.
\end{remark}

\begin{remark} \label{rmk:quintic} \rm The expansions of the quintics 
$a,b,c,d$ in Proposition  \ref{prop:coeffs}
have $36$ terms in $d_1,d_2,\ldots,d_6$.
 These are precisely the $D_4$-invariant quintics that are denoted by $Q$ in 
\cite[Example 2.6]{CD} and denoted by $F_1$ in \cite[Lemma 4.4]{CGL}.
\end{remark}

\begin{proposition} \label{prop:discriminant}
The discriminant of the cubic (\ref{eq:octanomial}) is the degree $32$ polynomial
\begin{equation}
\label{eq:discriminant}
 2^{16} 3^{5}\cdot e^2 f^2 g^2 h^2 \cdot (ac-eg)^2 \cdot (ad-fg)^2 \cdot (bc-eh)^2 \cdot (bd-fh)^2 \cdot
\Delta_{\mathcal{A}}, 
\end{equation}
where $\mathcal{A}$ is the point configuration (\ref{eq:support}) and $ \Delta_\mathcal{A}$  is the {\em $\mathcal{A}$-discriminant}
\cite[Chapter~9]{GKZ}. In our case, the $\mathcal{A}$-discriminant
has $49$ terms of degree $8$, namely 
$$ \footnotesize \begin{matrix}
\! \Delta_\mathcal{A} \,\,=\,\,
a^4 b^2 h^2-2 a^3 b^3 g h-2 a^3 b^2 c d h-2 a^3 b c f h^2-2 a^3 b d e h^2+4 a^3 e f h^3+a^2 b^4 g^2
{-}2 a^2 b^3 c d g \\
 {+}a^2 b^2 c^2 d^2 +8 a^2 b^2 c f g h+8 a^2 b^2 d e g h+4 a^2 b c^2 d f h
+4 a^2 b c d^2 e h-6 a^2 b e f g h^2+a^2 c^2 f^2 h^2 \\ -10 a^2 c d e f h^2  {+} a^2 d^2 e^2 h^2 
-2 a b^3 c f g^2-2 a b^3 d e g^2+4 a b^2 c^2 d f g+4
 a b^2 c d^2 e g-6 a b^2 e f g^2 h \\ 
  -2 a b c^3 d^2 f-2 a b c^2 d^3 e  - 10 a b c^2 
f^2 g h -26 a b c d e f g h-10 a b d^2 e^2 g h-2 a c^3 d f^2 h+8 a c^2 d^2 e f h \\ 
-2 a c d^3 e^2 h{+}18 a c e f^2 g h^2{+}18 a d e^2 f g h^2 {+} 4 b^3 e f g^3  +b^2 c^2 
f^2 g^2-10 b^2 c d e f g^2+b^2 d^2 e^2 g^2 \\
 -2 b c^3 d f^2 g+8 b c^2 d^2 e f g-2
 b c d^3 e^2 g + 18 b c e f^2 g^2 h  + 18 b d e^2 f g^2 h  +\,c^4 d^2 f^2-2 c^3 d^3 e f \\
+4 c^3 f^3 g h+c^2 d^4 e^2-6 c^2 d e f^2 g h-6 c d^2 e^2 f g h+4 d^3 e^3 g h-27 e^2 f^2 g^2 h^2.
\end{matrix}
$$
\end{proposition}

\begin{proof}
The resultant of four quaternary quadrics was written as the
determinant of a $20 \times 20$ matrix by Nanson \cite{Nan}.
We apply Nanson's formula to the four partial derivatives of 
(\ref{eq:octanomial}). The result is the expression shown in (\ref{eq:discriminant}).
The $\mathcal{A}$-discriminant is obtained by removing all factors in
(\ref{eq:discriminant}) that are supported on proper faces of 
the convex hull of $\mathcal{A} = 
\{ (1110), (1101), \ldots ,(0012) \}$.
\end{proof}

The study of the {\em Schl\"afli fan} in \cite{JPS}
 identified another family of cubic surfaces with eight terms.
The next theorem is analogous to the combinatorial results
 reported in \cite[Section 7]{JPS}, but now for the new octanomial model 
 (\ref{eq:octanomial}).

\begin{theorem} \label{thm:triangulations}
The configuration $\mathcal{A}$ is the vertex set
of a $3$-dimensional lattice polytope of normalized volume $7$.
This polytope has $70$ regular triangulations in $14$ symmetry classes under the action of $(\ZZ/2\ZZ)^3$.
Unimodular triangulations occur for $53$ of the $70$ triangulations. They come in $10$ symmetry classes:
$$ \footnotesize \begin{matrix}
\! \hbox{Orbit size} &  \hbox{Stanley--Reisner ideal} & \hbox{representative weights} &  \hbox{GKZ vector} \\
1 & \langle ah, bg, cf, de, eg, eh, fg, fh \rangle &   (4,1,7,2,9,5,9,9) & (5,5,5,5,2,2,2,2) \\
4 & \langle ab, ac, ah, cd, cf, eh, fg, fh \rangle &         (5, 2, 8, 4, 2, 9, 5, 9) &  (2, 5, 2, 5, 5, 2, 5, 2) \\ 
4 & \langle ab, ah, bg, cf, eg, eh, fg, fh \rangle &             (8, 3, 1, 1, 4, 9, 6, 9) &  (3, 3, 5, 7, 4, 2, 2, 2) \\ 
4 & \langle ab, ac, ah, bc, bg, cf, fg, fh, egh \rangle &        (9, 9, 6, 1, 4, 8, 8, 5) &  (2, 2, 3, 7, 6, 2, 3, 3) \\ 
4  & \!\!\! \langle ab, ac, ad, ah, bc, cd, cf, fh, bfg, deh\rangle &  (7, 3, 9, 1, 2, 4, 1, 9) &  (1, 4, 1, 4, 6, 3, 6, 3) \\ 
4 & \!\!\! \langle ab, ac, ad, ah, bc, bd, bg, cf, egh, fgh \rangle & (8, 8, 2, 6, 1, 6, 3, 7) &  (1, 1, 3, 5, 6, 4, 4, 4) \\ 
8 & \langle ab, ac, ah, bg, cf, eh, fg, fh \rangle &              (4, 2, 4, 1, 3, 8, 4, 4) &  (2, 3, 4, 7, 5, 2, 3, 2) \\ 
8 & \langle ab, ac, ad, ah, bg, cf, eh, fh \rangle &              (4, 4, 3, 1, 1, 1, 1, 7) &  (1, 3, 4, 6, 5, 3, 4, 2) \\ 
8 & \langle ab, ac, ad, ah, cd, cf, eh, fh, bfg \rangle &     (9, 3, 8, 4, 3, 9, 2, 9) &  (1, 5, 2, 4, 5, 3, 6, 2) \\ 
8 & \langle ab, ac, ad, ah, bc, bg, cf, fh, egh \rangle &   (9, 9, 4, 3, 4, 7, 3, 8) &  (1, 2, 3, 6, 6, 3, 4, 3) 
\end{matrix}
$$
\end{theorem}

\begin{proof}[Proof and explanation]
The vertex set of the  3-dimensional polytope ${\rm conv}(\mathcal{A})$ is $\mathcal{A}$ which we identify with $\{a,b,c,d,e,f,g,h\}$. This polytope has eight facets:
\begin{equation}
\label{eq:principalA} aceg \,,\quad
adfg  \,,\quad
bceh  \,,\quad
bdfh \,,\quad 
aef \,,\quad
bef \,,\quad
cgh \,,\quad
dgh .
\end{equation}
The derivation of its regular triangulations is a computation, using either an algebraic
approach or a geometric approach.
The following algebraic approach is based on \cite[Chapter 10]{GKZ}.
Namely, we compute the {\em principal $\mathcal{A}$-determinant}
\begin{equation} \label{eq:EA}
E_\mathcal{A} \,\,\, = \,\,\,
 a b c d \cdot e^2 f^2 g^2 h^2 \cdot (ac-eg) \cdot (ad-fg) \cdot (bc-eh) \cdot (bd-fh) \cdot \Delta_{\mathcal{A}}.   
\end{equation}
The GKZ vectors are the exponent vectors of the lowest monomials of 
(\ref{eq:EA}) with respect to all generic weight vectors $\upsilon \in \mathbb{R}^8$.
Representative weight vectors $\upsilon$ are shown in the third column.
The associated triangulation is a simplicial complex of dimension $3$ on
the vertex set $\mathcal{A}$. As is customary in combinatorial commutative
algebra \cite{GBCP}, we encode this simplicial complex by its {\em Stanley--Reisner ideal}, shown
in the second column. This squarefree monomial ideal is the initial ideal
$\,{\rm in}_\upsilon(I_\mathcal{A})\,$ of the {\em toric ideal} 
$I_\mathcal{A}$ associated with our configuration $\mathcal{A}$. 
\end{proof}

\begin{remark} The toric ideal of $\mathcal{A}$ is minimally generated by eight quadrics:
\begin{equation}
\label{eq:happytoric} \!\! I_\mathcal{A} \, = \,
\bigl\langle \,
\underline{ab}-cf \,, \,
\underline{ac}-eg \,, \,
\underline{ad}-fg \, , \,
\underline{ah}-cd \, ,\,
\underline{bg}-cd \, , \,
\underline{cf}-de\,, \,
\underline{eh}-bc\,, \,
\underline{fh}-bd \,
\bigr\rangle. \,
\end{equation}
We underlined the highest monomials with respect to
the weight vector $\upsilon$ in row~8 of the table above.
These monomials generate the initial ideal
$\,{\rm in}_\upsilon(I_\mathcal{A})$.
The eight generators 
form the reduced Gr\"obner basis
of $I_\mathcal{A}$ with respect to~$\upsilon$.
\end{remark}

\begin{remark}
When computing regular triangulations
algebraically, it is important to note is that the meaning of {\em initial monomial} depends
on the context. For a given weight vector $\upsilon$, we take 
lowest monomials in the principal $\mathcal{A}$-determinant $E_\mathcal{A}$ in order
to match taking highest monomials in the toric ideal~$I_\mathcal{A}$.
\end{remark}

Each line in $\PP^3$ is encoded by its Pl\"ucker coordinates
$(p_{01}:p_{02}:p_{03}:p_{12}:p_{13}:p_{23}) \in \PP^5$. To be precise, the affine cone over the
line is the image of the skew-symmetric $4 \times 4$-matrix $(p_{ij})$.
The universal Fano variety of  (\ref{eq:octanomial})  lives in
$\PP^5 \times \PP^7$. Its points are lines on octanomial cubic surfaces. 
We shall present two descriptions of these lines, first from the
perspective of equations (Proposition \ref{prop:ufv}),
and second in terms of the intrinsic del Pezzo geometry 
(Proposition \ref{prop:intrinsic}).

The ideal $I_{\rm ufv}$ of the universal Fano
variety is minimally generated by $1+20$ polynomials in
$\,\QQ[p_{01},p_{02},p_{03},p_{12},p_{13},p_{23},a,b,c,d,e,f,g,h]$.
This is implied by the results in \cite[Section 6]{JPS}.
The following census of the $27$ lines on our cubic surface was obtained  with {\tt Macaulay2}. It arises from the primary decomposition of $I_{\rm ufv}$ in the above polynomial ring with $14$ variables.

\begin{proposition} \label{prop:ufv}
The ideal $I_{\rm ufv}$ of the Fano variety in $\PP^5 \times \PP^7$
has $15$ minimal primes, $9 = 4+4+1$ of degree $1$, and
$6 = 2+4$ of degree~$3$.
The degree is the size of the fiber over the $\PP^7$ of octanomials~(\ref{eq:octanomial}).
Four of the $27$ lines are the coordinate lines
$\langle x,z \rangle$, $\langle x,w \rangle $, 
$\langle y,z \rangle$, $\langle y,w \rangle$.
Another four correspond to lines in coordinate planes in $\PP^3$, namely
$\langle x, dy {+} gz {+} hw \rangle$,
$\langle y, cx {+} gz {+} hw \rangle$,
$\langle z,ex {+} fy {+} bw \rangle$ and
$\langle w,ex {+} fy {+} az \rangle$.
One unique line  is disjoint from the coordinate lines. That line
has Pl\"ucker coordinates (\ref{eq:inthemiddle}).
Two triples of lines intersect pairs of coordinate lines,
and four triples of lines  intersect precisely one coordinate line each.
\end{proposition}

\begin{remark}
The line that is disjoint from the coordinate lines is defined by
$ \langle \,
(bg {-} ah)y + (fh {-} bd)z + (ad {-} fg)w\,, \,\,
(bg {-} ah)x + (eh{-}bc)z  + (ac{-} eg)w \,\rangle$. It has
\begin{equation}
\label{eq:inthemiddle}
\begin{matrix} \!\!\!\!\!
  (p_{01}:p_{02}:\cdots :p_{23})  \!\! & =  &\!\!
\bigl(\,bg{-}a h : fh {-} b d : a d {-} f g :  b c {-} e h:  eg {-} a c : c f {-} d e \,\bigr) \\
& = &
\hbox{maximal minors of} \,\,
\begin{tiny}
\begin{pmatrix} a & b & e & f  \\ g & h &  c & d
\end{pmatrix}.      \end{tiny}
\end{matrix}
\end{equation}
\end{remark}

Proposition \ref{prop:coeffs} gives a formula for the 
octanomial cubic over the field $K=\QQ(d_1,d_2,d_3,d_4,d_5,d_6)$.
The six points in (\ref{eq:sixpoints}) and the
cubics in (\ref{eq:factoredbasis}) are also defined over $K$.
Hence, so are the $27$ lines 
and their $135$ intersection points in $\PP^3$.
The census in Proposition \ref{prop:ufv} 
translates into formulas for these objects over $K$.

The middle four coordinates in (\ref{eq:inthemiddle}) are products of ten linear forms.
These come from the root system $E_6$. The first and last coordinates
are products of five such linear forms with one quintic as in Remark \ref{rmk:quintic}.
Here is another instance.

\begin{example} \label{ex:twotriples} \rm
Each line that intersects a pair of coordinate lines has
two zero Pl\"ucker coordinates.
The other four Pl\"ucker coordinates are products of roots.
For instance, one of the three lines with 
 $p_{02} = p_{13} = 0$ has the other coordinates
 $$ \footnotesize
\begin{matrix}
p_{01} = (d_5{-}d_6) (d_4{-}d_6) (d_3{-}d_5) (d_2{-}d_4) ,\quad & 
p_{03} =
 (d_4-d_6)^2 (d_3-d_5) (d_2-d_5),  \\
p_{12} = -(d_5-d_6)^2 (d_3-d_4)(d_2-d_4), \qquad & 
\quad p_{23} = (d_5{-}d_6)(d_4{-}d_6)(d_3{-}d_4)(d_2{-}d_5).
\end{matrix}
$$

Similarly, one of the three lines with $p_{03}=p_{12} = 0$ satisfies
$$  \footnotesize
 \begin{matrix}
p_{01} & = & (d_3{+}d_4{+}d_5)(d_2{+}d_4{+}d_5)(d_1{+}d_3{+}d_4)(d_1{+}d_2{+}d_5) , \\
p_{02} &= &(d_3{+}d_4{+}d_5)^2(d_1{+}d_2{+}d_5)(d_1{+}d_2{+}d_4)\,,\,\,\, \\
p_{13} & = & -(d_2{+}d_4{+}d_5)^2(d_1{+}d_3{+}d_5)(d_1{+}d_3{+}d_4)  , \\
p_{23} & = & -(d_3{+}d_4{+}d_5)(d_2{+}d_4{+}d_5)(d_1{+}d_3{+}d_5)(d_1{+}d_2{+}d_4).
\end{matrix}
$$
Formulas for all $27$ lines and their $135$ intersections are posted on our website.
\end{example}

\section{Arrangements of Trees} \label{sec3}

In this section we study the octanomial model over a field with valuation.
For concreteness, we work over the rational numbers $\QQ$
with the $p$-adic valuation, for some prime  $p \geq 5$.  We consider the open
surface obtained by removing the $27$~lines on a given cubic in $\PP^3$.
Ren, Shaw and Sturmfels \cite{RSS} studied the \emph{intrinsic tropicalization}
of such surfaces, and they identified     two generic types of tropical surfaces.
  They are characterized by their structure at infinity, which is an arrangement of
   $27$ trees with $10$ leaves \cite[Figures 4 and 5]{RSS}.
   These surfaces are points in the \emph{Naruki fan}, the tropical moduli space of 
   cubic surfaces in~\cite{HKT}. See  \cite[Table 1]{RSS} for a
  census of cones in this fan.
  Cueto and Deopurkar \cite{CD} realized the tree arrangements
  geometrically via a natural embedding into $\PP^{44}$.

For the following proposition we fix the  function field
$K = \QQ(d_1,\ldots,d_6 )$. The  cubic surface $S$ given by (\ref{eq:octanomial}) 
over the function field $K$
is smooth  in $\PP^3$ and has no Eckhart points.  The $27$ lines on $S$ are
labeled $E_1,\ldots,E_6$, $ F_{12}, \ldots,F_{56}$, $ G_1,\ldots, G_6$, as in  
\cite{CD, RSS}. Namely, the line $E_i$ is the exceptional fiber over the $i$-th point 
in (\ref{eq:sixpoints}), $F_{ij}$ is the
line connecting the $i$-th and $j$-th points, and $G_j$ is the conic
through the five points other than the $j$-th. Each line intersects $10$
 other lines. For instance, $E_i$ meets the five lines $F_{ij}$ and the five lines $G_j$ where $j \not= i$. 
These $27$ labels are now matched with the census of lines in Proposition~\ref{prop:ufv}.

\begin{proposition} \label{prop:intrinsic}
 The $27$ lines on the cubic (\ref{eq:octanomial}) are as follows.
We first have
$$  F_{12} = \{ x=z=0\}\,, \quad
F_{13} = \{ y=w=0 \} \,,\quad
F_{46}  = \{ y=z=0 \} \,,\quad
F_{56} = \{ x=w=0\}. 
$$
The other four lines that lie in coordinate planes are
$$
F_{34} \subset \{ x=0 \} \,, \quad
F_{25} \subset \{ y=0\} \, , \quad
F_{35} \subset \{ z=0 \} \, , \quad
F_{24} \subset \{ w=0\}.
$$
Two triples of lines (in Example \ref{ex:twotriples}) intersect a pair of coordinate lines, namely
$$ 
E_1,F_{45}, G_1 \,\,\hbox{ intersect } \,\, F_{12}, F_{13} \qquad {\rm and} \qquad
E_6,F_{23}, G_6 \,\,\hbox{ intersect } \,\, F_{46}, F_{56}.
$$
The  following four triples of lines intersect a unique coordinate line:
\begin{equation}
\label{eq:batch} \begin{matrix}
&E_2, F_{36}, G_2 \,\, \hbox{ intersect } \,\, F_{12} ,\,\,\,\,\,\,
E_3, F_{26}, G_3 \,\, \hbox{ intersect } \,\, F_{13}, \,\,\, \\
&E_4, F_{15}, G_4 \,\, \hbox{ intersect } \,\, F_{46} ,\,\,\,\,\,\, 
E_5, F_{14}, G_5 \,\, \hbox{ intersect } \,\, F_{56}.
\end{matrix}
\end{equation}
Finally, $F_{16}$ is the unique line (\ref{eq:inthemiddle}) 
that does not intersect any coordinate line.
\end{proposition}

\begin{proof}
The pullback of $\{x=0\} \cap S $ from $ \PP^3$
 to $\mathbb{P}^2$ is $ F_{12} F_{34} F_{56}$ 
and the pullback of $\{z=0\} \cap S$ 
is $ F_{12}F_{35}F_{46}$. The line $F_{12}$ is common to both and thus equals
$\{x=z=0\}$. The line $F_{34}$ appears only in $\{x=0\}$.  The lines $E_1,F_{45}, G_1$ all meet
$ F_{12}$ and $ F_{13}$ but are not in a coordinate plane; e.g. $F_{45}$ does not appear as a factor in
(\ref{eq:factoredbasis}).
The argument for (\ref{eq:batch}) is similar. Hence, $F_{16}$ is the match  for
(\ref{eq:inthemiddle}).
\end{proof}

Pairs of intersecting lines determine $135$ points on $S \subset\mathbb{P}^3$.
Using the rational formulas in Section~\ref{sec2}, we precomputed these points over 
$K = \QQ(d_1,d_2,\ldots,d_6)$. This allows for rapid evaluation
when the $d_i$ are specialized to rational numbers.

\smallskip

We now describe a general construction 
of $27$ metric trees  which is  essential in the theory  of
tropical del Pezzo surfaces \cite{CD, RSS}.
Consider a smooth cubic surface $S \subset \PP^3$,
defined over a valued field. For instance,  we could take
$\QQ$ with its $p$-adic valuation and let $S$
be the octanomial surface defined by
$d_1,\ldots,d_6  \in \QQ$.
For each of the lines on $S$, we fix an isomorphism with $\PP^1$ by
projecting to one of the six coordinate axes in $\PP^3$.
The $10$ points on that line are now the columns of a $2 \times 10$ matrix with entries in $\QQ$.
We record the $p$-adic valuations of the $45$ maximal minors of this matrix.
This data determines a phylogenetic tree with $10$ leaves as in \cite[Section 4.3]{MS}.
In our implementation we use the method of {\em Quartet Puzzling} \cite{BS01}
for constructing the  tree from its $45$ pairwise distances.

Each interior edge in one of our trees
is encoded by a list of $\leq 7$ splits.
A {\em split} is a partition of the ten leaf labels into two subsets that is induced
by removing an interior edge from the tree.
Let $s_i$ denote the number of splits into
$i$ leaves versus $10-i$ leaves, where $i \in \{2,3,4,5\}$.
The string $[s_2 s_3 s_4 s_5]$ is a combinatorial
invariant of the tree.
The {\em tree statistic} of
our arrangement is the multiset of these $27$
strings. The notation $[s_2 s_3 s_4 s_5]^u$ means that the
string $[s_2 s_3 s_4 s_5]$ arises from $u$ of the $27$ trees.
We illustrate the use of this notation with an example.

\begin{example} \rm \label{ex:twotypes} 
Ren et al.~\cite{RSS} identified 
two types, denoted  (aaaa) and (aaab),
 of  tree arrangements that are generic in the Naruki fan.
 The
drawings  in \cite[Figures 4 and 5]{CD} or in
\cite[Figures 4 and 5]{RSS}
show that
their tree statistics are
 $\bigl\{[4021]^{24} ,[4020]^3 \bigr\}$ for (aaaa),
 and $\bigl\{[2221]^{12},  [4201]^{12}, [4210]^3 \bigr\}$ for~(aaab). 
\end{example}

We now come to the main results  in this section.
A choice of moduli parameters $ (d_1,d_2,d_3,d_4,d_5,d_6)$ in $\QQ^6$
is called \emph{Naruki general} if  the resulting tree arrangement 
is of type (aaaa) or of type (aaab). We say that the octanomial~(\ref{eq:octanomial}) is 
{\em tropically smooth} if the induced polyhedral
subdivision of its Newton polytope ${\rm conv}(\mathcal{A})$
is one of the $53$ unimodular triangulations in Theorem~\ref{thm:triangulations}.

 Tropical smoothness implies
classical smoothness for hypersurfaces in $\PP^n$ with full support.
This can be deduced from \cite[Proposition 4.5.1]{MS}.
However, sparse hypersurfaces are typically singular in $\PP^n$
even if they are tropically smooth in the sense above.
It is a key feature of our octanomial model, based on delicate combinatorics,
that the two notions of smoothness are compatible.

 \begin{theorem} \label{thm:smoothsmooth}
 Every tropically smooth octanomial is classically smooth in $\PP^3$.
 \end{theorem}
 
\begin{proof}
  The secondary polytope of $\mathcal{A}$ is the
  Newton polytope of the principal $\mathcal{A}$-determinant,
  by \cite[Theorem 10.1.4]{GKZ}.
  The principal $\mathcal{A}$-determinant
  (\ref{eq:EA}) has the same (non-monomial) irreducible
  factors as the discriminant (\ref{eq:discriminant}) of the cubic~(\ref{eq:octanomial}).
  Hence, the two polytopes have the same normal fan. This is the 
  secondary fan of $\mathcal{A}$. 
  Since the surface (\ref{eq:octanomial}) is tropically smooth, the vector
  $\,{\rm val}(a,b,\ldots,h)\,$ lies in the interior
  of a maximal cone of this fan. This means that the initial form
  of the discriminant (\ref{eq:discriminant}) picked by this weight vector is a monomial.
  In particular, the discriminant is nonzero, and
  we conclude that the surface is smooth.
\end{proof}

\begin{remark}
It would be interesting to identify a combinatorial characterization of all 
configurations $\mathcal{A}$ in $\ZZ^d$
for which the conclusion of Theorem \ref{thm:smoothsmooth} holds. Put differently, under which 
conditions on the support set $\mathcal{A}$ does the
tropical smoothness of a hypersurface in $\RR^d$ imply tropical smoothness in $\mathbb{TP}^d$?
\end{remark}

We saw in Proposition \ref{prop:ufv} that eight of the $27$ lines on an octanomial surface
lie in coordinate planes in $\PP^3$.
 Therefore, the tropicalizations of these lines are in the boundary of
 the tetrahedron that represents the tropical projective space $\mathbb{TP}^3$. 
We refer to \cite[Chapter 3]{MR} for the polyhedral construction of tropical toric varieties
like $\mathbb{TP}^3$.
The following is the octanomial version of Conjecture~\ref{conj:kristin}.

\begin{theorem}\label{thm:distinct_lines}
Fix a prime $p \geq 5$ and an octanomial cubic surface $S$ defined over $\QQ_p$.
If $S$ is tropically smooth then the $27$ lines on $S$ have 
distinct tropicalizations in $\mathbb{TP}^3$.
In that case, it follows that all $27$ lines on $S$ are defined over $\QQ_p$.
\end{theorem}

\begin{proof} We use normalized Pl\"ucker coordinates
where the first nonzero entry is~$1$.
We claim that the coordinates of the $27$ lines
have distinct valuations. By Proposition~\ref{prop:ufv}, nine of the lines
are identified by their zero coordinates. The other lines come in
six triplets. Each triplet is identified by its zero coordinates.
   We must show that the tropical lines from the same triplet are distinct.

For the lines in a triplet, each nonzero Pl\"ucker coordinate $p_{ij}$ lies in a degree three extension 
of $ \QQ(a,b,\ldots,h)$. From  the minimal primes in Proposition~\ref{prop:ufv}, we compute
the minimal polynomial of $p_{ij}$. This gives six sets of 
irreducible polynomials in  $\ZZ[a,b,\ldots,h][t]$ that are cubic in $t$.
The first two sets have size three. The other four sets have size four.
All coefficients in $\ZZ[a,b,\ldots,h]$ have their own coefficients in $\{\pm 1, \pm 2\}$. 
These scalars have valuation $0$ since $p \geq 5$.

As (\ref{eq:octanomial}) is tropically smooth, ${\rm val}(a,b,\ldots,h) \in C$,
where $C \subset \RR^8$ is the Gr\"obner cone (cf.~\cite[Chapter 8]{GBCP}) for one of the $53$
squarefree initial ideals of $I_\mathcal{A}$. 
This containment translates into
linear inequalities among valuations. For instance, for the reduced
Gr\"obner basis in (\ref{eq:happytoric}), we see that
${\rm val}(a,b,\ldots,h) \in C$ if and only~if
\begin{equation}
\label{eq:happytoric2}
\begin{small}
\begin{matrix}
 {\rm val}(a) + {\rm val}(b) > {\rm val}(c) + {\rm val}(f), & 
{\rm val}(a) + {\rm val}(c) > {\rm val}(e) + {\rm val}(g),  \\
 {\rm val}(a) + {\rm val}(d) > {\rm val}(f) + {\rm val}(g), & 
 {\rm val}(a) + {\rm val}(h) > {\rm val}(c) + {\rm val}(d),  \\
 {\rm val}(b) + {\rm val}(g) > {\rm val}(c) + {\rm val}(d), & 
 {\rm val}(c) + {\rm val}(f) > {\rm val}(d) + {\rm val}(e),  \\
{\rm val}(e) + {\rm val}(h) > {\rm val}(b) + {\rm val}(c), &
{\rm val}(f) + {\rm val}(h) > {\rm val}(b) + {\rm val}(d) . &
\end{matrix}
\end{small}
\end{equation}
We now assume that this holds, i.e.~${\rm val}(a,\ldots,h) $ is in one the $53$ Gr\"obner cones.

Take one of the univariate cubics $P=c_3t^3+c_2t^2+c_1t^2+c_0$ obtained above. 
In each $c_i \in \ZZ[a,b,\ldots,h]$ we
look for a monomial whose valuation equals $v_i = {\rm val}(c_i)$,
assuming the Gr\"obner cone inequalities. 
The three roots of $P$ give a
 coordinate $ p_{ij}$ for a triplet of lines.
Their valuations are distinct if and only if
\begin{equation}
\label{eq:distinctroots}
 v_0 + v_2 > 2 v_1 \quad {\rm and} \quad v_1 + v_3 > 2 v_2 . 
 \end{equation}
Our strategy is to show that the cone inequalities like 
(\ref{eq:happytoric2}) imply those in (\ref{eq:distinctroots}).

This strategy would fail for some $P$  if one of its coefficients $c_i$ does not have a unique 
monomial of minimal valuation $\nu_i$, or if the four valuations do not satisfy (\ref{eq:distinctroots}).
If this  happens then we  must discard $P$ and move on 
to the next univariate cubic in the same list
(of three or four cubics). For each of the six lists,
 it suffices to identify one cubic $P$ in that list having distinctly valued roots.
 
We applied this strategy to each of the $10$ representative initial ideals
in Theorem \ref{thm:triangulations}
and to each of the six lists of univariate cubics. For each triangulation and each triplet,
we found that there is a cubic $P$ that works.
\end{proof}

A fundamental problem regarding cubic surfaces over a valued field is to understand
the relationship between their extrinsic and intrinsic tropicalizations.
This is the motivation for the embedding into $\PP^{44}$ studied in \cite{CD}.
We are now aiming to take this back into $\PP^3$. For the octanomial model,
we show that genericity can occur simultaneously on both the extrinsic
and the intrinsic side.

 \begin{proposition}  \label{prop:HorrayForMarta}
 For at least five of the ten combinatorial types 
     in Theorem  \ref{thm:triangulations},
there exists a Naruki general vector
$ d = (d_1,d_2,\ldots,d_6) \in \mathbb{Q}^6$
whose corresponding octanomial cubic is 
tropically smooth and has this combinatorial type.
 \end{proposition}

At present, we do not know whether the other five types are realizable.
We derived Proposition \ref{prop:HorrayForMarta} by an extensive computation, based on sampling
Naruki general points $d$ from $\QQ^6$. The following sampling method was used.
The tropical moduli space, which is the $4$-dimensional Naruki fan,
is the image of a tropical linear space of dimension $5$.
This is the uniformization $\,{\rm Berg}({\rm E}_6) \rightarrow {\rm trop}(\mathcal{Y}^0)\,$ in the
second row of \cite[Equation (3.1)]{RSS}. The domain is 
the Bergman fan associated with 
the root system ${\rm E}_6$. This is a $5$-dimensional
fan in $\mathbb{R}^{35} \simeq \mathbb{R}^{36}/\mathbb{R} {\bf 1}$.
The number of cones in these fans are reported in \cite[Lemma 3.1]{RSS}.

Let $\mathbb{E}_6 $ denote the matrix in $\{0,1\}^{6 \times 36}$ whose columns
are the $36$ roots of~${\rm E}_6$. The map $d \mapsto {\rm val}(d \cdot \mathbb{E}_6)$ takes
$\QQ^6$ onto the Bergman fan ${\rm Berg}({\rm E}_6)$. Here we are  referring to the fan structure on ${\rm Berg}({\rm E}_6)$ described by Ardila  et al. in \cite{ARW}. We seek general points
in the $142,560$  maximal cones of ${\rm Berg}({\rm E}_6)$. To find them,
we select an ordered column basis of $\mathbb{E}_6$.
We pick $e = (e_1,e_2,e_3,e_4,e_5,e_6) \in~\QQ^6$ such that
${\rm val}(e_1) < {\rm val}(e_2) < \cdots < {\rm val}(e_6)$.
The row vector $d$ is obtained by multiplying the inverse of the
$6 \times 6$ submatrix of $ \mathbb{E}_6$ on the left by $e$.
The basis specifies a chain of flats in the matroid of ${\rm E}_6$,
and  ${\rm val}(d \cdot \mathbb{E}_6)$ lies in the maximal cone of ${\rm Berg}({\rm E}_6)$
indexed by that chain~(cf.~\cite[Theorem 4.2.6]{MS}).  With this choice, $d$ is
likely to land in a maximal Naruki cone under $\,\QQ^6 \rightarrow {\rm Berg}({\rm E}_6) \rightarrow {\rm trop}(\mathcal{Y}^0)$.

\begin{proof}[Proof of Proposition \ref{prop:HorrayForMarta}]
Fix prime $p \geq 5$. The following two vectors  are Naruki general of type (aaaa) and  their octanomial surfaces are tropically smooth:
$$
\begin{small}
\begin{matrix}
(2+p^5-p^7-p^9, \, -p^3+p^9, \, -1+ p^7, \, -p^3-p^7+p^9+p^{11}, \, -1+p^9,  \, 1+p^3-p^9), \smallskip \\
(1+p^3-p^9, \, p^3+p^5-p^9, \, -2+p^5+p^7+p^9-p^{11}, \, 1-p^3-p^5+p^{11}, \, -p^3+p^9, \, -1+p^9).
\end{matrix}
\end{small}
$$
The resulting octanomial surfaces are tropically smooth.
The corresponding triangulations appear in lines  1 and 2
in the classification in Theorem~\ref{thm:triangulations}.

Likewise, the following three
integer vectors $(d_1,\ldots,d_6)$ are Naruki general and
their tree arrangements have type (aaab):
$$
\begin{small}
\begin{matrix}
\big(-1+p^7, \, 2+p^5-p^7-p^9, \, -p^3+p^9, \, -1+ p^9, \, 1+p^3-p^9, \, 1+p^3-2 p^9+p^{11}\big), \smallskip \\
\big(-1+p^3-p^5+p^7\,,\, \, 2-p^7+p^9-p^{11}\,,\, \, 2-p^3+p^5-p^7+p^9-p^{11}, \\
-1+p^7-p^9+p^{11}\,,\, \,\, -2+p^3+p^7-p^9+p^{11}\,,\,\, \, -1+p^{11}\big), \smallskip\\
\big(2-p^5-p^7+p^9, \, 2-p^3-p^5+p^9, \, -1+p^5+p^7-p^9, \, -1+p^3, \, -1+p^5, \, -1+p^3-p^9+p^{11}\big). 
\end{matrix}
\end{small}
$$
The resulting octanomial surfaces are tropically smooth.
The corresponding triangulations appear in lines  3, 4 and 7
in the classification in Theorem~\ref{thm:triangulations}.

At present, we do not know whether any
tropically smooth octanomial surface of types 
5, 6, 8, 9, 10  can have a Naruki general tree arrangement.
\end{proof}

We now come to octanomial cubics that are not Naruki general.
In all examples that follow we work over $\mathbb{Q}$
with the $p$-adic valuation for  $p=5$.
We begin with cubic surfaces that are {\em stable}, in the sense discussed in
\cite[Section 5]{RSS}. These correspond to the lower-dimensional cones
in the Naruki fan ${\rm trop}(\mathcal{Y}^0)$. Their $27$ trees
are obtained from those in (aaaa) or (aaab) by contracting some interior edges. 
The various stable non-generic types are listed in \cite[Table 1]{RSS}.

\begin{example} The following moduli vectors $d \in \QQ^6$
define tree arrangements that are non-generic but stable.
  We indicate the type as denoted in \cite[Table 1]{RSS}. 
$$ 
\footnotesize 
\begin{matrix} 
(d_1, d_2, d_3, d_4, d_5, d_6) & \hbox{Tree arrangement}  && \hbox{Type}  \\
 (2377, -2375, 1240, 2385, 2425, 2625) & \begin{matrix}
 \bigl\{ [2210]^1, [2220]^4, [2221]^8 ,\\ \qquad \qquad
 [4201]^{12} \,, \, \, [4210]^2 \bigr\} \end{matrix} && (aab) 
 \smallskip \\
(-843, 124, 724, 744, 1537, 844) &
 \bigl\{ [2020]^1,  [4020]^6,   [4021]^{20}  \bigr\}
&& (aaa)
\end{matrix}
$$
\end{example}

Our last family of instances is the most
interesting and mysterious one.
These $ d \in \QQ^6$
give tree arrangements that do not appear in \cite[Table 1]{RSS}.
The corresponding cubic surfaces are {\em not stable}.
This means that the fiber of the vertical map
$\,{\rm trop}(\mathcal{G}^0) \rightarrow {\rm trop}(\mathcal{Y}^0)\,$ in
\cite[Equation (3.1)]{RSS} has dimension $3$.
The tropical cubic surfaces arising from these $d$ are contained in that fiber.
The combinatorial structure is hence not revealed by the analysis in \cite[Section~3]{RSS}.

\begin{example} \label{ex:threeseven}
The following vectors  $d\in \QQ^6$ determine cubics that are not stable:
$$
\footnotesize \begin{matrix}
(d_1, d_2, d_3, d_4, d_5, d_6) & \hbox{Tree arrangement} \smallskip \\
\!\! (-719, 1081, -359, -347, -9287, 10081) \!\!\!\! \! & 
\{ [2220]^6, \, [3210]^3, \, [3220]^6, \, [4201]^{12} \} 
\smallskip \\
(120, -3099, -3095, 620,-595, 3100) & 
\!\!\!\!\!\!
\{
 [2220]^2\!, [2221]^8 \!, [3210]^1 \! ,  [3220]^2 \! ,[4201]^{12} \!, [4210]^2 \} 
\smallskip
\\
(-6719, 1248,  7248, -519,481,  -479)& 
\{ [3020]^1,\, [4020]^4 ,\,  [4021]^{20}, \,
[5020]^2 \} \\
\end{matrix}
$$
To see that these arrangements are not stable,
we note that any edge contraction in a tree would lead 
an entry $s_i$ in $[s_2s_3s_4s_5]$ to decrease. Therefore the trees 
$[3220]$ and $[5020]$ appearing above do not arise from either (aaaa) or (aaab).
\end{example}

To understand examples such as these, one needs to go much beyond~\cite{RSS}.
This was accomplished by Cueto and Deopurkar in their remarkable article~\cite{CD}.
In \cite[Proposition 4.4]{CD}, they explain that the tree arrangement is determined
by the valuations of the Cross functions. Each Cross function is a difference of
Yoshida functions, and it factors into four roots and a quintic as in Remark \ref{rmk:quintic}.
Non-stable tree arrangements arise because of cancellations for specific $d_1,\ldots,d_6 \in \mathbb{Q}$.
If the $p$-adic valuation of a Cross function is not predicted by combinatorics
then Table~1 in \cite{RSS} does not apply.
The phenomenon is explained in \cite[Section 10]{CD}, where a detailed
explanation of non-stable trees and their edge lengths is given. For example,
consider the tree of type $[5,0,0,0]$ shown in \cite[Figure 6]{CD}. This tree explains
the initial entry ``5'' in the last type in Example~\ref{ex:threeseven}.

We thank Angelica Cueto for explaining these results
to us and for confirming the correctness of Example~\ref{ex:threeseven}
by analyzing the surfaces in $\PP^{44}$.
At present we do not know which combinatorial types
of non-stable tree arrangements are realizable
over fields such as $\mathbb{Q}_p$.
This will be the topic of a subsequent project.

\section{Dense Cubics} \label{sec4}

We now turn to cubic surfaces in $\PP^3$ whose defining polynomial has full support:
\begin{equation} \label{eq:cubicf}
\begin{small}
\begin{matrix}
  & c_0 w^3 + c_1w^2z + c_2wz^2 + c_3z^3 + c_4w^2y + c_5wyz + c_6yz^2 \\ & + \,
c_7wy^2 + c_8y^2z  + c_9y^3  +  c_{10}w^2x  +
c_{11}wxz + c_{12}xz^2 + c_{13}wxy \\ & + \,
c_{14}xyz + c_{15}xy^2 + c_{16}wx^2 + c_{17}x^2z + c_{18}x^2y +
c_{19}x^3.
\end{matrix}
\end{small}
\end{equation}
This uses notation as in \cite{JPS}. 
Our primary goal is to state a conjecture on the arithmetic of the $27$ lines on a
tropically smooth cubic surface over a complete valued field such as $\mathbb{Q}_p$.
This generalizes Theorem \ref{thm:distinct_lines}.
For the most part, Section \ref{sec4} is independent of the previous sections.
We no longer study the octanomial model.
Only at the very end, we return to the title of this paper, by describing   an algorithm 
for transforming (\ref{eq:cubicf}) into octanomial normal form~\eqref{eq:octanomial}.

For ease of exposition we assume that the $c_i$ are general, so no line in the cubic
meets any of the six coordinate lines in $\PP^3$. Hence, all $27$ lines
lie in the six standard affine charts $\{p_{ij} \not = 0\}$ of the Grassmannian. 
To compute the $27 $ lines, we substitute $\,w = s x + t y \,$ and $\, z = u x + vy\,$ into (\ref{eq:cubicf})
where $s,t,u,v$ are unknowns. The result is a binary cubic in $x,y$ whose coefficients define the Fano scheme:
$$
\begin{small}
\begin{matrix}
c_0 t^3 + c_1 t^2 v+c_2 t v^2+c_3 v^3+c_4 t^2+c_5 t v+c_6 v^2+c_7 t+c_8 v+c_9
\qquad \qquad \qquad 
& = & 0 , \\
c_0 s^3+c_1 s^2 u+c_2 s u^2+c_3 u^3+c_{10} s^2+c_{11} s u+c_{12} u^2+c_{16} s+c_{17} u+c_{19}
\qquad \qquad  
&=& 0, \\
3 c_0 s t^2+2 c_1 s t v+c_2 s v^2+c_1 t^2 u+2 c_2 t u v+3 c_3 u v^2+2 c_4 s t+c_5 s v 
+c_{10} t^2 
& &  \\ 
\qquad \qquad 
+c_5 t u+c_{11} t v+2 c_6 u v+c_{12} v^2+c_7 s+c_{13} t+c_8 u+c_{14} v+c_{15} & = & 0, \\
%\end{matrix}
%\end{small}
%$$
%
%$$
%\begin{small}
%\begin{matrix}
3 c_0 s^2 t+c_1 s^2 v+2 c_1 s t u+2 c_2 s u v+c_2 t u^2+3 c_3 u^2 v+c_4 s^2
+2 c_{10} s t  +c_5 s u & & \\ 
\qquad \qquad 
+c_{11} s v+c_{11} t u+c_6 u^2+2 c_{12} u v+c_{13} s+c_{16} t
+c_{14} u+c_{17} v+c_{18} & = & 0 .
\end{matrix}
\end{small}
$$
These four cubic equations in four unknowns $s,t,u,v$ have
 $27$ distinct solutions over the algebraic closure of
$K = \mathbb{Q}(c_0,c_1,\ldots,c_{19})$.
One can try to solve this symbolically (e.g.~in {\tt Magma}), 
but this is  unpractical for dense polynomials~\eqref{eq:cubicf}.

Using combinatorial methods \cite{GRZ}, we can instead compute the set of
$p$-adic valuations  ${\rm val}(s,t,u,v) $ $ \in \mathbb{Q}^4$
of the $27$ solutions $(s,t,u,v) \in \overline{\mathbb{Q}}\vphantom{\mathbb{Q}}^4$.
Moreover, the $p$-adic series expansion of these four scalars up to some desired
order can also be found. To do this, we use the implementation of $p$-adic arithmetic in    {\tt Magma}. 
     The feasibility of this approach is underscored by Conjecture \ref{conj:kristin} below.
   
     Smooth tropical surfaces, with full Newton polytope, come in  $14,373,645$ combinatorial types.  We used the database presented in \cite{JPS} to sample from tropically smooth cubics, and to conduct experiments that support Conjecture~\ref{conj:kristin}. In each case, we also computed the $135$ intersection points in $\PP^3$ over $\mathbb{Q}_p$ and we built the arrangement of $27$ trees using Quartet Puzzling as in Section \ref{sec3}.
  
\begin{conjecture} \label{conj:kristin}
  If the surface (\ref{eq:cubicf}) over $\QQ_p$ is tropically smooth then its lines have distinct tropicalizations.
In particular, the $27$ lines in $\PP^3$, their $135$ intersection points, and the $6$ points in $\PP^2$ obtained by blowing down six skew lines are all defined over 
the $p$-adic field $\mathbb{Q}_p$. The algebraic closure of $\mathbb{Q}_p$ is not needed.
\end{conjecture}

Kristin Shaw announced a proof that every
tropically smooth family of complex cubic surfaces contains
$27$ lines whose tropical limits are distinct (personal communication, 2019).
The approach is based on a correspondence theorem between tropical~\cite{Shaw} and complex intersection theories that is proved using tropical homology~\cite{IKMZ}. 
Shaw's result, if true  in the $p$-adic setting,  would imply Conjecture \ref{conj:kristin}.
 We suspect that the  conjecture holds over all complete discretely valued fields. 
Our Theorem~\ref{thm:distinct_lines} gives further evidence in this direction.

\smallskip

We next explain how to find six points in $\PP^2$ and a basis for the
space of cubics through these points that yields the given cubic (\ref{eq:cubicf}).
We implemented the following method in {\tt Magma}. 
Our code runs fast because computations are done using floating point arithmetic in $\mathbb{Q}_p$.
This fact is based on Conjecture~\ref{conj:kristin}.

Let $S \subset \PP^3$ be the cubic surface in (\ref{eq:cubicf}). Arguing as above, 
we find its $27$ lines $p$-adically. Fix
six pairwise skew lines $E_1,\dots,E_6 \subset S$.
Label the other $21$ lines on $S$ as follows. The line
$F_{ij}$ intersects $E_i$ and $E_j$ but no other $E_k$, 
and $G_i$ intersects  $E_j$ for $j \in \{1,\ldots,6\}\backslash \{i\}$.
 We will compute a morphism $\pi :  S \to \mathbb{P}^2$ that
contracts $E_1,\ldots,E_6$ to points. The map $\pi$ is specified uniquely by requiring
$$ \pi(E_1)=(1:0:0)\,,\,\,\,
     \pi(E_2)=(0:1:0)\,,\,\,\,
     \pi(E_3)=(0:0:1)\,,\,\,\,
      \pi(E_4)=(1:1:1) . $$
      For $i \neq j$, let $H_{ij} $ be the plane spanned by $G_i$ and $E_j$, 
      and let $h_{ij}$ be a linear form defining $H_{ij}$.
      For any point $q\in \mathbb{P}^3\setminus G_i$, let $H_{iq}$ be the plane spanned by $G_i$ and~$q$.  

\begin{theorem}\label{thm:project}
  Let $U_{ij}=S\setminus ( G_i \cup G_j \cup F_{ij} )$. 
  Then $S=U_{12} \cup U_{13} \cup U_{23}$, and the blow-down map 
  $\pi  :  S \rightarrow \PP^2$ is given
  on each chart by a quadratic map as follows:
  \begin{eqnarray*}
    \pi|_{U_{12}}(q) &=&  
 \bigl( \,u_{12} \cdot h_{12}(q)h_{23}(q) \,:\, v_{12} \cdot h_{21}(q)h_{13}(q) \,:\, w_{12}\cdot h_{12}(q)h_{21}(q)\, \bigr), \\
    \pi|_{U_{13}}(q) &=&  
 \bigl(\, u_{13} \cdot h_{13}(q)h_{32}(q) \,:\, v_{13} \cdot h_{13}(q)h_{31}(q) \, :\, w_{13} \cdot h_{31}(q)h_{12}(q) \,\bigr),\\
    \pi|_{U_{23}}(q) &=&  
\bigl( \, u_{23} \cdot h_{23}(q)h_{32}(q) \,: \,v_{23} \cdot h_{23}(q)h_{31}(q) \,:\, w_{23} \cdot h_{32}(q)h_{21}(q) \, \bigr).
  \end{eqnarray*}
  Here $u_{ij},v_{ij},w_{ij} $ are nonzero constants that are determined by $\pi(E_4)=(1:1:1)$.
\end{theorem}

\begin{proof}
The equation $S=U_{12} \cup U_{13} \cup U_{23}$
is seen by examining the incidences among the lines $F_{ij}$ and $G_k$
that appear in the definition of the open sets $U_{ij}$.

We now prove the formula for $\pi$ on $ U_{12}$. The other two cases follow upon relabeling.
Let $\gamma_1: \PP^3 \dashrightarrow F_{23} $ denote the projection from the line $G_1$,
and let $\gamma_2: \PP^3 \dashrightarrow F_{13} $ denote the projection from $G_2$.
Note that $\gamma_i(E_j)$ is a point if $j \in \{1,2,3,4,5,6\} \backslash \{i\}$.
We identify $F_{23}$ with $\PP^1$ so that
$\gamma_1(E_2) = (1:0)$, $\gamma_1(E_3) = (0:1)$ and $\gamma_1(E_4)=(1:1)$. Likewise,
we identify $F_{13}$ with $\PP^1$ so that
$\gamma_2(E_1) = (1:0)$, $\gamma_2(E_3) = (0:1)$ and $\gamma_1(E_4)=(1:1)$. This
implies
 $\gamma_1(q) = (h_{13}(q): \mu \cdot h_{12}(q) )$ and $\gamma_2(q) = (h_{23}(q): \nu  \cdot h_{21}(q) ) $ where $\mu$ and $\nu$ are constants. The asserted formula for $ \pi|_{U_{12}}(q) $ is obtained by fusing these two projections $S \dashrightarrow \PP^1$.

The blow-down map $\,\pi : S \rightarrow \PP^2\,$ has the property that
$\gamma_1|_S$ is the composition of $\pi$ followed by 
$\PP^2 \dashrightarrow \{ X=0\}  \simeq \PP^1$, and
$\gamma_2|_S$ is $\pi$ followed by 
$\PP^2 \dashrightarrow \{ Y=0\}  \simeq \PP^1$.
We see from our construction that 
$F_{12},F_{13},F_{23}$ are mapped to the three coordinate lines in
$\PP^2$ and that $E_1,E_2,\ldots,E_6$ are mapped to points.
The images  in $\PP^2$ of the other $18$ lines are determined by their intersection patterns 
on $S \subset \PP^3$. In particular, the point $\pi(E_4)$ lies in $\{XYZ \not= 0\}$. This completes the proof.
\end{proof}

The blowup map $\varphi :  \mathbb{P}^2 \dashedrightarrow S \subset \mathbb{P}^3$ 
is defined by a tuple $(g_0,g_1,g_2,g_3)$ of ternary cubics. The cubic curve cut out by $g_i$ 
is the image under $\pi$ of the intersection of $S$ with the $i$-th coordinate plane in $\PP^3$. 
Using the formulas for $\pi$ above, we may compute the vanishing locus of $g_i$ 
and therefore find a cubic $\tilde g_i = \lambda_i g_i$. In practice, we carry
 this out by interpolation using the six base points and the images of the 
 points on $S$ lying on coordinate lines. Since we assumed that $S$ is
  tropically smooth, its points on coordinate lines are all defined over $\QQ_p$.

To find the scalars $\lambda_0,\ldots,\lambda_3$, we consider the map $\tilde\varphi :  \mathbb{P}^2 \dashedrightarrow \mathbb{P}^3$ defined by $(\mu_0\tilde g_0,\dots,\mu_3\tilde g_3)$ with indeterminates $\mu_i$. The additional constraint that the image of $\tilde\varphi$ must lie in $S$ reveals the 
entries of $(\mu_0 , \ldots ,\mu_3)=(\lambda_0^{-1}, \ldots , \lambda_3^{-1})$.

\smallskip

We now turn to Question~11 from the $27$ questions:
{\em How to construct six points  with integer coordinates in $\PP^2$,
and a basis for the space of cubics  through these points,
such that the resulting cubic surface in $\PP^3$ has
a smooth tropical surface for its $p$-adic tropicalization?
Which unimodular triangulations arise?}

The following theorem, which is conditional on our conjecture,
would give the definitive answer to Question 11. It implies that every
smooth tropical cubic arises, and hence so does each of the
$14,373,645$ unimodular triangulations.

\begin{theorem} \label{thm:eleven}
Suppose Conjecture \ref{conj:kristin} holds. Fix any 
tropically smooth cubic (\ref{eq:cubicf}) whose coefficients $c_i$ are $p$-adic numbers.
Then there exist six points in $\PP^2$ and a basis for their cubics, both defined over
the rational numbers $\QQ$,
such that the resulting classical cubic surface in $\PP^3$ has the same tropicalization as (\ref{eq:cubicf}).
\end{theorem}

\begin{proof}
Two classical cubics in $\PP^3$ have the same tropicalization if the valuations of 
their coefficients coincide.
This holds if the distance between their coefficient vectors in $\QQ_p^{20}$ is small 
with respect to the  $p$-adic supremum norm.

Let $S$ be the tropically smooth cubic in~\eqref{eq:cubicf}.
Let  $\xp=(p_1,\dots,p_6)$ and $\xg=(g_0,\dots,g_3)$ be the six points and four ternary cubics constructed by blowing down $S$ as above. We claim that $\xp$ and $\xg$ can be approximated with rational coefficients so that the new cubic defined by these approximations is close to~\eqref{eq:cubicf}.

   If Conjecture~\ref{conj:kristin} holds then the six points $\xp$ are defined over $\QQ_p$. Therefore, we can find six rational points $\xq=(q_1,\dots,q_6)$ approximating $\xp$ to any specified degree of $p$-adic accuracy. Let $V$ be the space of ternary cubics passing through $\xq$ and let $\tilde \xg = (\tilde g_0,\dots,\tilde g_3) $ be orthogonal projections of the cubics in $\xg$ onto $V$. 
 Orthogonal projection can be defined in the $p$-adic setting, and the distances 
 $||g_i-\tilde g_i||$ are comparable to the distances $||p_i-q_i||$ 
in an explicit fashion. We can and do choose $\tilde \xg$ to have rational coefficients.
Let  $\tilde S$ be the resulting cubic in $\PP^3$.
  
 The coefficient vectors of the cubic surfaces $S$ and $\tilde S$ span the kernels of 
   the following $19 \times 20$ matrices $M$ and $ \tilde M$.
       Choose $19$ general points in $ \PP^2_{\QQ}$, evaluate the maps 
   $ \PP^2 \to \PP^3$ given by $\xg$ and $\tilde \xg$ on these points, and then evaluate 
  all cubic monomials in $x,y,z,w$ on the image points  to get the columns of $M$ and $\tilde M$.
    
    By Cramer's rule, the coefficients $c_i,\tilde c_i$ of $S,\tilde S$ are the maximal minors of
     $M,\tilde M$. Since $\xg$ and $\tilde \xg$ are arbitrarily close, so are the entries of 
     the matrices $M$ and $\tilde M$. The
     cubic forms that define the surfaces $S$ and $\tilde S$ are now arbitrarily close 
     in the $p$-adic supremum norm. Hence, $S$ and $\tilde S$ have the same
     tropicalization.
\end{proof}

The proof of Theorem~\ref{thm:eleven} translates into the following algorithm
for answering Question~11. 
We begin by choosing one of the 
$14,373,645$ distinguished representatives in \cite{JPS} for the
  smooth tropical cubic surfaces. 
Next we choose a classical cubic $S$ with full support (\ref{eq:cubicf}) having that tropicalization.
We then check that the $27$ lines have distinct tropicalizations
as asserted by Conjecture~\ref{conj:kristin}. 
We compute $(\xp,\xg)$
and a rational approximation
$(\tilde \xp, \tilde \xg)$
as explained in the proof.
 The final step is to verify that $S$ and $\tilde S$ have the same
tropicalization. We implemented this algorithm in {\tt Magma}.
It is available at our website.

\begin{example} \label{ex:type7} \rm
We illustrate the algorithm
for  the combinatorial
  type \# 7 in \cite{JPS}.
   The representative coefficient vector for this type of tropical cubic surface is
  $$   ({\rm val}(c_0),\ldots,{\rm val}(c_{19})) \,\,=\,\,   (16,7,3,11,9,2,5,3,0,0,5,0,4,3,0,2,6,5,8,15).  $$
  We fix $p=5$ and choose the canonical lifts
  $  c_0 = 5^{16}, \,c_1 = 5^7\! ,\,c_2 = 5^3\!,\, \ldots, \,c_{19} = 5^{15}$. 
   We compute the $27$ lines in this cubic,  and we find that it is Naruki general.
  
 We identify six skew lines $E_1,\dots,E_6$ and write the  map $\pi$ as in Theorem~\ref{thm:project}.
The six points in $\PP^2$ are $p_i=\pi(E_i)$, with $p_1,p_2,p_3,p_4$ in standard position.
For the other two points $p_5,p_6$, the first few terms 
in their $p$-adic expansions~are
%[
%[ 5^-1*(4*5^0 + 4*5^1 + 1*5^2 + 3*5^3 + 2*5^5 + ...), 5^5*(4*5^0 + 4*5^1 + 4*5^2 + 2*5^4 + 1*5^5 + ...), 1],
%[ 5^7*(1*5^0 + 4*5^1 + 1*5^2 + 1*5^3 + 2*5^4 + ...), 5^5*(4*5^0 + 4*5^1 + 4*5^2 + 4*5^3 + 4*5^4 + ...), 1]
%]
$$ 
\begin{small}
  \begin{matrix}
    p_5  &=&  \bigl( \,5^{-1}\!\cdot (4 + 4\cdot 5^1 + 1\cdot 5^2 + 3\cdot 5^3 + \dots) \,\,:\,\,\, 
      5^{5}\! \cdot (4 + 4\cdot 5^1 + 4\cdot 5^2 + 0\cdot 5^3 + \dots)\,\, :\,\, 1\,\bigr), \\
    p_6 &=& \bigl(\,5^{7}\!\cdot (1 + 4\cdot 5^1 + 1\cdot 5^2 + 1\cdot 5^3 + \dots) \,\,:\,\,\,
      5^{5}\!\cdot (4 + 4\cdot 5^1 + 4\cdot 5^2 + 4\cdot 5^3 + \dots) \,\,:\,\, 1\,\bigr).
\end{matrix}
\end{small}
$$
  We next replace these by nearby points over $\QQ$. We set $q_i=p_i$ for $i=1,2,3,4$. We round $p_5$ and $p_6$ to rational points while retaining $14$ digits of precision:
$$ \begin{matrix}
  q_5 &=& \bigl( \,2473616049/5 \,\,:\,\, -425393750 \,\,: \,\, 1 ), \\
  q_6 &=& \bigl( \,1331718750 \,\,:\,\, -2324221875 \,\,:\, \, 1 \,\bigr).
\end{matrix} $$
  
  The special $\mathbb{Q}_p$-basis $\xg = (g_0,\ldots,g_3)$ of cubics 
  through $\xp = (p_1,\ldots,p_6)$ is found as described 
  after Theorem~\ref{thm:project}.  Fix
  any $\mathbb{Q}$-basis for the space $V$ of cubics 
  through $\xq = (q_1,\ldots,q_6)$. We project
each  $g_i$ into $V$ via the non-archimedean version of the Gram--Schmidt process
\cite[Section 2.3]{UZ}. This step is done over~$\mathbb{Q}_p$.
The image cubics are  rounded to cubics over $\mathbb{Q}$
while staying in $V$ and preserving the distance to $\xg$. We used a variation of $p$-adic LLL explained in~\cite{IN} to find good rational approximations of $p$-adic vectors.
The result is $\tilde \xg = (\tilde g_0,\ldots,\tilde g_3)$. 
The cubic surface
  $\tilde S \subset \PP^3$ is the image of the map $\xg$.
  Its rational coefficients $\tilde c_i $ are very big. Their valuations match those of 
  the $c_i$ we started with. For instance,
  $$
  \begin{small}
  \begin{matrix}
  \tilde c_0 \!\!\! & = & \!\!\!\!
5^{16}\cdot (1 + 3\cdot 5^{1} + 4\cdot 5^{2} + 3\cdot 5^{3} + 2\cdot 5^{4} + 4\cdot 5^{5} + 2\cdot 5^{6} + 4\cdot 5^{7} + 1\cdot 5^{9} + 3\cdot 5^{12} \quad \\
& & + 3\cdot 5^{13} + 1\cdot 5^{14} + 1\cdot 5^{16} + 1\cdot 5^{18} + 4\cdot 5^{19} + 1\cdot 5^{20} + 4\cdot 5^{21} + 2\cdot 5^{22} + 4\cdot 5^{24} + \dots  \quad \\ 
& & \dots +  3\cdot 5^{209} + 4\cdot 5^{210} + 2\cdot 5^{212} + 2\cdot 5^{213} + 2\cdot 5^{214} + 1\cdot 5^{215} + 4\cdot 5^{216} + 2\cdot 5^{217}).
\end{matrix}
\end{small}
$$
This concludes our derivation of an explicit example for answering Question 11.
\end{example}

We close Section~\ref{sec4} by returning to the octanomial model.
We explain how to derive, for a given cubic, the moduli coordinates $d_i$ and thus the normal form  (\ref{eq:octanomial}).
This requires us to identify a cuspidal cubic through our six points in $\PP^2$
and to transform that cubic into the standard form $\{X^2 Z = Y^3\}$.
From this we can then read off the matrix (\ref{eq:sixpoints}). We carry out this
computation in {\tt Magma} as follows.

Our input is six points $p_1,\dots,p_6$ in $\PP^2$ over a field $K$.
There is a web of cubic curves passing through them. We choose a seventh 
point $p_7$ to cut down the dimension and obtain a net 
$\mathcal{N} \simeq \PP^2$ of cubics. If the points are general enough 
then $\mathcal{N}$ contains $24$ cuspidal cubics. Hence, 
the product space $\mathcal{N} \times \PP^2$ contains $24$ 
pairs $(f,r)$ where $r$ is a cusp on the cubic $\{f=0\}$.
These $24$ points are defined by bihomogeneous equations represented by the
gradient and the $2\times 2$ minors of the Hessian of $f \in \mathcal{N}$. Our 
equations are $\,\nabla f(r) = 0 \,$ and  $\,{\rm rank} \bigl({\rm He}(f)(r)\bigr) =1$.

Extending to $\overline{K}$, we pick one solution $(f,r)$. The cubic curve $\{f=0\}$
passes through $p_1,\ldots,p_7$ and has a cusp at $r$.
In addition, it has a unique smooth inflection point $r'$. Let $\ell_0,\ell_1,\ell_2$ denote linear forms defining the cuspidal tangent, the line through $r$ and $r'$, and the inflection tangent at $r'$ such that $f=\ell_1^3-\ell_0^2\ell_2$.

The triple $(\ell_0,\ell_1,\ell_2)$ defines the 
automorphism of $\PP^2_{\overline{K}}$ that puts the cubic $C$
 into standard form $X^2Z=Y^3$. In particular, the $E_6$ moduli 
 can now be read off:
\[
  \qquad  d_i \,\,\, = \,\,\,  {\ell_1(p_i)}/{\ell_0(p_i)} \qquad
  {\rm for} \,\,\,\, i = 1,2,\ldots, 6. 
\]
We explain how to implement this method in our setting,
where the given cubic (\ref{eq:cubicf}) has rational coefficients $c_i$.
We first identify the points
 $p_1,\ldots,p_6$ in $\PP^2_K$, where  $K = \mathbb{Q}_p$ as above.
 For computing the $24$ solutions $(f,r)$ in
 $\mathcal{N} \times \PP^2_K$, one would like to use
    Gr\"obner bases. But, this requires
 special care because the polynomials to be solved have
  numerical coefficients, namely series in $\mathbb{Q}_p$.
   
 Instead, we work with the rational approximations $q_1,\dots,q_6 \in \PP^2_\QQ$
 computed above. We also choose $q_7 \in \PP^2_\QQ$.
    A cuspidal curve through $q_1,\ldots,q_7$  typically does not exist over $\QQ$.
    But, it often exists over $\QQ_p$, depending on the choice of~$q_7$.
    We always succeeded with this after several tries. Now, the pair
    $(f,r)$ has been found over $\QQ_p$.   
     The rest of the computation is
    linear algebra over $\QQ_p$.
    Namely,  we compute     $\ell_0$ as the tangent line to $f$ at $r$.
    The Hessian of $f$ equals $ \ell_0^2\ell_1$
    times a constant, so we can find $\ell_1$ up to constant.
Next, the inflection point $r'$ is found by intersecting the
curve $f$ with line $\ell_1$. Finally $\ell_2$ is the inflection line at $r'$.
The constants are now found from the desired equation
$\,f=\ell_1^3-\ell_0^2\ell_2$.

We applied this method to the rational 
configuration $q=(q_1,\dots,q_6)$ in Example~\ref{ex:type7}.
From the resulting cuspidal cubic $f$, we read the moduli parameters 
$$
\begin{small}
\begin{matrix}
d_1 & = &   1, \\
d_2 & = &   2 + 2\cdot 5^1 + 2\cdot 5^2 + 3\cdot 5^3 + 2\cdot 5^4 + 1\cdot 5^5 + 2\cdot 5^6 + 4\cdot 5^9 \dots,\\
d_3 & = &   2 + 4\cdot 5^1 + 4\cdot 5^2 + 4\cdot 5^3 + 2\cdot 5^4 + 4\cdot 5^5 + 4\cdot 5^6 + 2\cdot 5^7 + 4\cdot 5^8 + 1\cdot 5^9 + \dots,\\
d_4 & = &   2 + 2\cdot 5^1 + 2\cdot 5^2 + 3\cdot 5^3 + 4\cdot 5^4 + 3\cdot 5^5 + 4\cdot 5^6 + 3\cdot 5^9 + \dots,\\
d_5 & = &   1 + 2\cdot 5^4 + 2\cdot 5^5 + 1\cdot 5^6 + 2\cdot 5^7 + 1\cdot 5^8 + \dots,\\
d_6 & = &   1 + 3\cdot 5^1 + 2\cdot 5^2 + 1\cdot 5^3 + 4\cdot 5^4 + 3\cdot 5^5 + 4\cdot 5^6 + 1\cdot 5^7 + 4\cdot 5^8 + 1\cdot 5^9 + \dots 
\end{matrix}
\end{small}
$$
The full details on this example, and on all others, are found on our website.

\begin{small}

\end{small}

\end{document}